\documentclass[11pt]{article}
\usepackage{a4wide}

\usepackage[a4paper, margin=2.3cm]{geometry}
\usepackage{amsmath,amssymb,amsthm}
\usepackage{color}
\usepackage{parskip}
\usepackage{hyperref}
\usepackage{parskip}
\usepackage{setspace}
\usepackage{graphicx}

\newtheorem{theorem}{Theorem}[section]
\newtheorem{remark}{Remark}[section]
\newtheorem{corollary}[theorem]{Corollary}
\newtheorem{lemma}[theorem]{Lemma}

\newtheorem*{definition*}{Definition}

\def\E{E}

\begin{document}
\title{Distribution of distances in five dimensions\\ and related problems}
\author{François Clément\thanks{ETH Zurich, Switzerland. Email: fclement@student.ethz.ch}\and Thang Pham\thanks{The group Theory of
Combinatorial Algorithms, ETH Zurich. Email: vanthang.pham@inf.ethz.ch}}
\date{}
\maketitle
\begin{abstract}
In this paper, we study the Erd\H{o}s-Falconer distance problem in five dimensions for sets of Cartesian product structures. More precisely, we show that for $A\subset \mathbb{F}_p$ with $|A|\gg p^{\frac{13}{22}}$, then $\Delta(A^5)=\mathbb{F}_p$. When $|A-A|\sim |A|$, we obtain stronger statements as follows: 
\begin{enumerate}
\item if $|A|\gg p^{\frac{13}{22}}$, then $(A-A)^2+A^2+A^2+A^2+A^2=\mathbb{F}_p.$
\item if $|A|\gg p^{\frac{4}{7}}$, then $(A-A)^2+(A-A)^2+A^2+A^2+A^2+A^2=\mathbb{F}_p.$
\end{enumerate}
We also prove that  if $p^{4/7}\ll |A-A|=K|A|\le p^{5/8}$, then
\[|A^2+A^2|\gg \min \left\lbrace \frac{p}{K^4}, \frac{|A|^{8/3}}{K^{7/3}p^{2/3}}\right\rbrace.\]
As a consequence, $|A^2+A^2|\gg p$ when $|A|\gg p^{5/8}$ and $K\sim 1$, where $A^2=\{x^2\colon x\in A\}$.
\end{abstract}
\section{Introduction}
Let $q=p^n$ be an odd prime power, and $\mathbb{F}_q$ be the finite field of order $q$. For any two points $x=(x_1, \ldots, x_d)$ and $y=(y_1, \ldots, y_d)$ in $\mathbb{F}_q^d$, the algebraic distance between them is defined by the formula: 
\[||x-y||=(x_1-y_1)^2+\cdots +(x_d-y_d)^2.\]
Let $E$ be a set in $\mathbb{F}_q^d$, we denote the set of distances determined by pairs of points in $E$ by $\Delta(E)$. The Erd\H{o}s-Falconer distance problem asks for the smallest number $\alpha>0$ such that for any $E\subset \mathbb{F}_q^d$ with $|E|\gg q^{\alpha}$, we have $\Delta(E)=\mathbb{F}_q$, or $|\Delta(E)|\gg q$. 

We use the following notations in this paper: we write $X \ll Y$ if  there exists an absolute constant $K>0$ such that $X \leq K Y$, $X \sim Y$ if $X\ll Y$ and $Y \ll X$, and $X \gtrsim Y$ if there exists an absolute constant $K'>0$ such that $X \gg (\log(Y))^{-K'}Y$.

Iosevich and Rudnev \cite{io} proved that for any dimension $d$ we have $\alpha\le \frac{d+1}{2}$ by using discrete Fourier analysis. Hart, Iosevich, Koh, and Rudnev \cite{hart} showed that, in general over arbitrary finite fields, the exponent $\frac{d+1}{2}$ is optimal in odd dimensions. It is conjectured in even dimensions that $\alpha=\frac{d}{2}$.  In a recent paper, Murphy, Petridis, Pham, Rudnev, and Stevens \cite{murphy} established that for any $E\subset \mathbb{F}_p^2$, if $|E|\gg p^{5/4}$, then $|\Delta(E)|\gg p$. This improves the 10-year-old exponent $\frac{4}{3}$ given by Chapman, Erdogan, Hart, Iosevich and Koh in \cite{CEHIK10} over arbitrary finite fields. 

Using  Rudnev's point-plane incidence bound \cite{RR}, Pham and Vinh \cite{PV} proved that when $E$ is of Cartesian product structures, then, to cover all possible distances, the exponent $\frac{d+1}{2}$ can be improved. More precisely, they obtained the following theorem. 
\begin{theorem}[\cite{PV}]\label{ppv}
Let $E=A^d\subset \mathbb{F}_p^d$. Suppose that $d\ge 6$, then there exist $\epsilon_d=\frac{3 \cdot 2^{\frac{d-5}{2}}-(\frac{d+1}{2}) }{3 \cdot 2^{\frac{d-3}{2}}-1}$ for d odd, and $\epsilon_d=\frac{2^{\frac{d}{2}}-d-1}{2^{\frac{d}{2}+1}-2} $ for d even, such that if $|A^d|=|A|^d\gtrsim p^{\frac{d+1}{2}-\epsilon_d}$, then  $\Delta(A^d)=\mathbb{F}_p$. 
\end{theorem}
\begin{corollary}[\cite{PV}]\label{1.6}
For $A\subset \mathbb{F}_p$ with $|A|\gtrsim p^{4/7}$, we have $\Delta(A^6)=\mathbb{F}_p$. 
\end{corollary}
In the first theorem of this paper, we show that the exponent $\frac{d+1}{2}$ can also be improved in five dimensions. 
\begin{theorem}\label{thm1}
For $A\subset \mathbb{F}_p$ with $|A|\gg p^{13/22}$, then we have 
\[\Delta(A^5)=(A-A)^2+(A-A)^2+(A-A)^2+(A-A)^2+(A-A)^2=\mathbb{F}_p.\]
\end{theorem}
We remark here that this theorem is the finite field analogue of a recent result on the Falconer distance problem in the continuous setting by Koh, Pham, and Shen in \cite{KPS}, namely, for $A\subset \mathbb{R}$ of Hausdorff dimension at least $13/22$, then the distance set $\Delta(A^5)=\{|x-y|\colon x, y\in A^5\}$ has non-empty interior, where $|x|$ is the Euclidean norm. In higher dimensions, the same conclusion holds under the condition
$$ \dim_H(A) > \left\{ \begin{array}{ll} \frac{d+1}{2d} \quad &\mbox{if}~~ 2\le d\le 4,\\
\frac{d+1}{2d}-\frac{d-4}{2d(3d-4)}\quad &\mbox{if}~~ 5\le d\le 26,\\
\frac{d+1}{2d}-\frac{23d-228}{114d(d-4)}\quad &\mbox{if}~~ 27\le d. \end{array} \right.$$
Notice that these dimensional thresholds are bigger than the corresponding sizes of sets in the finite field analogue (Theorem \ref{ppv}) when $d\ge 6$.

Since the distance function is invariant under translations, we can always assume that $0\in A$. If $|A-A|\sim |A|$, then we are able to improve Theorem \ref{thm1} further as follows. 
\begin{theorem}\label{thm1.333}
Let $A\subset \mathbb{F}_p$. Suppose that $|A-A|\sim |A|$ and $|A|\gg p^{13/22}$, then we have 
\[(A-A)^2+A^2+A^2+A^2+A^2=\mathbb{F}_p.\]
\end{theorem}
\begin{theorem}\label{thm1.444}
Let $A\subset \mathbb{F}_p$. Suppose that $|A-A|\sim |A|$ and $|A|\gg p^{4/7}$, then we have
\[(A-A)^2+(A-A)^2+A^2+A^2+A^2+A^2=\mathbb{F}_p.\]
\end{theorem}
The main ingredient in the proofs of these improvements is the next result which is interesting on its own and  says that the size of $A^2+A^2$ is quite large when $|A-A|\sim |A|$, where $A^2:=\{x^2\colon x\in A\}$. 
\begin{theorem}\label{thm2}
Let $A$ be a set of $\mathbb{F}_p$ with $|A-A|=K|A|$.

\begin{enumerate}
\item If $K|A|\ll p^{2/3}$, then \[|A^2+A^2|\gg \min \left\lbrace \frac{p}{K^4}, \frac{|A|^{19/8}}{K^{21/8}p^{1/2}}\right\rbrace.\]
\item If $p^{4/7}\ll K|A|\le p^{5/8}$, then we have a better bound
\[|A^2+A^2|\gg \min \left\lbrace \frac{p}{K^4}, \frac{|A|^{8/3}}{K^{7/3}p^{2/3}}\right\rbrace.\]
\end{enumerate} 
In particular, when $K\sim 1$: 
\begin{enumerate}
\item if $|A|\gg p^{5/8}$, then we have $|A^2+A^2|\gg p$.
\item if $p^{\frac{4}{7}+\epsilon}\le |A|\le p^{\frac{2}{3}-\epsilon}$ for some $\epsilon>0$, then $|A^2+A^2|\gg |A|^{\frac{3}{2}+\epsilon'}$ with $\epsilon'=\epsilon'(\epsilon)>0$. 
\end{enumerate}
\end{theorem}
It is worth noting that a similar theorem was obtained by Iosevich, Koh and Pham for very small sets in \cite{IKP}, namely, when $|A||A-A||A^2-A^2|\le p^2$ and $|A-A|=|A|^{1+\epsilon}$, $0<\epsilon<1/54$, then $|A^2-A^2|\gtrsim |A|^{1+\frac{9-27\epsilon}{17}}$. Moreover, the approach in their paper does not work for the case of  $A^2+A^2$. 

We also remark that our proofs of Theorem \ref{thm1} and Theorem \ref{thm2} rely on recent results on \textit{bisector energies} and distance sets due to Murphy et al. in \cite{murphy}.

The rest of this paper is organized as follows: Proofs of Theorem \ref{thm1} and Theorem \ref{thm2} are presented in Section 2 and 3, respectively. We discuss variants of Theorem \ref{thm2} in Section 4. The last section is devoted for proofs of Theorems \ref{thm1.333} and \ref{thm1.444}.
\section{Proof of Theorem \ref{thm1}}
To prove Theorem \ref{thm1}, we recall the following theorems. The first is a point-line incidence bound due to Vinh \cite{vinhline}, and the second is a distance result of sets of medium size due to Murphy et al. \cite{murphy}.
\begin{theorem}[Theorem 3, \cite{vinhline}]\label{ptsplane}
Let $P$ be a set of points and $L$ be a set of lines in $\mathbb{F}_p^2$. Then the number of incidences between $P$ and $L$, denoted by $I(P, L)$, satisfies
\[\left\vert  I(P, L)-\frac{|P||L|}{p} \right\vert\le p^{1/2}\sqrt{|P||L|}.\]
\end{theorem}
\begin{theorem}[Theorem 2, \cite{murphy}]\label{thmzx}
Let $E$ be a set in $\mathbb{F}_p^2$. Suppose that $4p<|E|\le p^{5/4}$, then\begin{equation}\label{dt}|\Delta(E)|\gg \frac{|E|^{4/3}}{p^{2/3}}.\end{equation}
\end{theorem}
%For any $t\in \mathbb{F}_p$, let $\nu(t)$ be the number of pairs $(x, y)\in E\times E$, $E\subset \mathbb{F}_p^2$, such that $||x-y||=t$. As a consequence of Theorem \ref{thmzx}, by the Cauchy-Schwarz inequality, we have 
%\begin{equation}\label{eqnorm}\sum_{t\in \mathbb{F}_p}\nu(t)^2\ll p^{2/3}|E|^{8/3}.\end{equation}
%Using the Cauchy-Schwarz inequality one more time, we have 
%\begin{equation}\label{dt}|\Delta(E)|\gg \frac{|E|^{4/3}}{p^{2/3}}.\end{equation}
With these results in hand, we are ready to prove Theorem \ref{thm1}. 
\begin{proof}[Proof of Theorem \ref{thm1}]
Let $\lambda$ be an arbitrary element in $\mathbb{F}_p$. We now show that if $|A|\gg p^{13/22}$, then there exist $x, y\in A^5$ such that $||x-y||=\lambda$. 

It is enough to show that under the condition on the size of $A$, the following equation has at least one solution: 
\begin{equation}\label{eqxx090}(x-y)^2+u+v=\lambda,\end{equation}
where $x, y\in A$, $u, v\in \Delta(A^2)$. 

Let $P=\{(-2x, v+x^2)\colon x\in A, v\in \Delta(A^2)\}$ be a point set in $\mathbb{F}_p^2$. Let $L$ be the set of lines of the form 
\[yX+Y=\lambda-u-y^2,\]
where $y\in A, u\in \Delta(A^2)$. 

It is not hard to see that the number of solutions of the equation (\ref{eqxx090}) is equal to the number of incidences between $P$ and $L$ in $\mathbb{F}_p^2$. 

Theorem \ref{ptsplane} tells us that whenever $|P||L|\gg p^3$, then there is at least one incidence between $P$ and $L$. 

On the other hand, we know that $|P|=|A|\cdot |\Delta(A^2)|=|L|$. Thus,  to conclude the proof, we only need the condition 
\[|A|^2\cdot |\Delta(A^2)|^2\gg p^3,\]
which is satisfied by the fact that $|A|\gg p^{13/22}$ and the inequality (\ref{dt}) with $E=A\times A$. 
\end{proof}
\begin{remark}
Compared to the approach of Theorem \ref{ppv} in \cite{PV}, our proof is much shorter. For any $t\in \mathbb{F}_p$, let $\nu(t)$ be the number of pairs $(x, y)\in E\times E$, $E\subset \mathbb{F}_p^2$, such that $||x-y||=t$. It follows from computations in \cite[pages 15, 17]{murphy}, we know that $\sum_{t\ne 0}\nu(t)^2$, which is $\sum_{r\ne 0}|S_r|^2$ in \cite{murphy}, is at most $p^{2/3}|E|^{8/3}$.  With this estimate, to obtain the number of pairs $(x, y)\in A^5\times A^5$  of a given distance, one can follow identically the proof in \cite{PV} to show that for any $A\subset \mathbb{F}_p^5$ and $\lambda\in \mathbb{F}_p$ with $p^{13/22}\lesssim |A|\le p^{5/8}$, the number of pairs of distance $\lambda$ is at least $(1+o(1))\frac{|A|^{10}}{p}$. We need to mention that in the proof of Pham and Vinh, an upper bound for $\sum_{t\ne 0}\nu(t)^2$, instead of $\sum_{t\in \mathbb{F}_p}\nu(t)^2$, would be sufficient.
\end{remark}
\begin{remark}
In dimensions $d\le 4$, the above argument breaks down, and even implies an exponent which is bigger than $\frac{d+1}{2}$. If for any two points $x$ and $y$ we consider the Minkowski distance function between them, namely, $||x-y||_M:=(x_1-y_1)^2-(x_2-y_2)^2$, then Rudnev and Wheeler \cite{RJ} proved that for any $A\subset \mathbb{F}_p$ with $|A+A|, |A-A|\le K|A|<\sqrt{p}$, the number of pairs in $A\times A$ of a given distance is at most $O(K^{6/5}|A|^{29/10})$.
\end{remark}
\section{Proof of Theorem \ref{thm2}}
To prove Theorem \ref{thm2}, we need a variant of Theorem \ref{ptsplane} for multi-sets. A proof can be found in \cite[Lemma 14]{Hanson}. 
\begin{theorem}[Lemma 8, \cite{Hanson}]\label{point-line}
Let $P$ be a multi-set of points in $\mathbb{F}_p^2$ and $L$ be a multi-set of lines in $\mathbb{F}_p^2$. The number of incidences between $P$ and $L$ satisfies
\[I(P, L)\le \frac{|P||L|}{p}+p^{1/2}\left(\sum_{u\in \overline{P}}m(u)^2  \right)^{1/2}\cdot \left(\sum_{\ell\in \overline{L}}m(\ell)^2  \right)^{1/2},\]
where $\overline{X}$ is the set of distinct elements in the multi-set $X$ and $m(x)$ is the multiplicity of $x$, $|X|=\sum_{x\in \overline{X}}m(x)$. 
\end{theorem}
\subsection{Bisector energy of a set}
Let $A\subset \mathbb{F}_p$, and $E:=A\times A$. For any two points $a=(a_1, a_2),$ $b=(b_1, b_2)$, the bisector line of the segment $ab$ is defined by the equation 
\[||x-a||=||x-b||.\]
This line is called isotropic if $||a-b||=0$, and non-isotropic otherwise. 

Notice that there might exist pairs $(a, b)\in E\times E$ and $(a', b')\in E\times E$ such that they have the same bisector line. Let $L_B$ be the multi-set of bisector lines determined by pairs of points in $E$, and $\overline{L}_B$ be the set of distinct bisector lines determined by pairs of points in $E$, and for each $\ell\in \overline{L}_B$, let $m(l)$ be its multiplicity. The quantity $\sum_{\ell\in \overline{L}_B}m(\ell)^2$ is called the bisector energy of $\overline{L}_B$. When the size of $E$ is not too big, we have the following lemma, which is a summary of some results from \cite{murphy, petridis}. 
\begin{lemma}\label{bisectorenergy}
Suppose $|A|\le p^{2/3}$, then we have 
\begin{equation}\label{eqener1}\sum_{\ell\in \overline{L}_B, ~\ell~ \mathtt{non-isotropic}}m(\ell)^2\ll |A|^{21/4}.\end{equation}
In addition, when $p^{4/7}\ll |A|\ll p^{5/8}$, we have a better bound, namely, 
\begin{equation}\label{eqener2}\sum_{\ell\in \overline{L}_B, ~\ell~ \mathtt{non-isotropic}}m(\ell)^2\ll p^{1/3}|A|^{14/3}.\end{equation}
\end{lemma}
\begin{proof}
For any $r\in \mathbb{F}_p$, let $S_r$ be the number of pairs $(x, y)\in E\times E$ such that $||x-y||=r$. It has been proved in \cite[Proposition 12]{murphy} that 
\[\sum_{\ell\in \overline{L}_B, ~\ell~ \mathtt{non-isotropic}}m(\ell)^2\ll M|E|+\sum_{r\ne 0}|S_r|^{3/2},\]
where $M$ is the maximal number of points from $E$ on a circle or on a line. Since $E=A\times A$, we can bound $M\le |A|$. 
Using the Cauchy-Schwarz inequality and the fact that $\sum_{r}|S_r|=|E|^2$, one has 
\[\sum_{\ell\in \overline{L}_B, ~\ell~ \mathtt{non-isotropic}}m(\ell)^2\ll |A|^3+|A|^2\cdot \left(\sum_{r\ne 0}|S_r|^2\right)^{1/2}.\]
Note that $\sum_{r\ne 0}|S_r|^2$ is equal to the number of tuples $(x, y, z, t)\in E\times E\times E\times E$ such that $||x-y||=||z-t||$. This can be bounded by at most  $|E|$ times the number of isosceles triangles in $E$, as a consequence of the Cauchy-Schwarz inequality. Note that here we only need to count the triangles $(a, b, c)\in E\times\times E$ with $||a-b||=||a-c||\ne 0$. There will be three types of these isosceles triangles: $||b-c||=0$, $||b-c||\ne 0$, and $b=c$. 

It has been proved in \cite{petridis} that if $|A|\le p^{2/3}$, then the number of isosceles triangles in $A\times A$ is at most $\ll |A|^{9/2}$. Thus, 
\[\sum_{\ell\in \overline{L}_B, ~\ell~ \mathtt{non-isotropic}}m(\ell)^2\ll |A|^2\cdot \bigg(|A|^2\cdot |A|^{9/2}\bigg)^{1/2}=|A|^{21/4}.\]

When $p^{4/7}\ll |A|\ll p^{5/8}$, by a direct computation, the number of isosceles triangles of the form $(a, b, c)$ with $||a-b||=||a-c||\ne 0$ is at most $|E|^2$ for $b=c$, at most $3|E|^2$ for $||b-c||=0$. Moreover, it has been proved in \cite[Proposition 15]{murphy} that the number of isosceles triangles with $||b-c||\ne 0$ in $E=A\times A$ is at most $p^{2/3}|A|^{10/3}$. So the above argument gives us the desired result.
\end{proof}

\begin{remark}
We note that  one can adapt the methods from \cite{IKP, KPS2} to prove that the bisector energy is at most  $|A-A|$ times the number of collinear triples in $A\times A$, which is bounded by $|A-A|\cdot |A|^{9/2}$. This is slightly weaker than the bound of Lemma \ref{bisectorenergy} when $|A-A|\sim |A|$. 
\end{remark}
\subsection{Proof of Theorem \ref{thm2}}
Set $D=A-A$.
We now consider the equation 
\[u=(x+y)^2+(z+t)^2,\]
where $x, z\in D, y, t\in A, u\in A^2+A^2\setminus \{0\}$. 

Let $N$ be the number of solutions of this equation. 
It is not hard to see that $N\ge |A|^4-2|A|$.
 
On the other hand, by the Cauchy-Schwarz inequality, we have 
\[N\le \sqrt{|A^2+A^2|}\cdot \E^{1/2},\]
where $\E$ is the number of tuples $(d_1, d_2, d_3, d_4, a_1, a_2, a_3, a_4)\in D^4\times A^4$ such that 
\[ (d_1+a_1)^2+(d_2+a_2)^2=(d_3+a_3)^2+(d_4+a_4)^2\ne 0.\]
By the Cauchy-Schwarz inequality, $\E$ can be bounded by at most $|A\times A|\cdot T$, where $T$ is the number of isosceles triangles $(x, y, z)\in (-A\times -A)\times (D\times D)\times (D\times D)$ such that $||x-y||=||x-z||\ne 0$.

Let $T_1$ be the number of isosceles triangles with $||y-z||\ne 0$, and $T_2$ be the number of isosceles triangles with  $y=z$ or $||y-z||=0$. A direct computation implies $T_2\le 4|A|^2|D|^2$. We now bound $T_1$. 

Let $L_B$ be the multi-set of bisector lines determined by pairs of points in $D\times D$.  We observe that $T_1$ is equal to the number of incidences between points in $-A\times -A$ and non-isotropic lines in $L_B$.

We now fall into two cases: 

{\bf Case $1$:} Assume $|D|\le p^{2/3}$.  If we use the estimate (\ref{eqener1}), namely, 
\[\sum_{\ell\in \overline{L}_B, ~\ell ~\mathtt{non-isotropic}}m(\ell)^2\ll |D|^{21/4},\]
then, applying Theorem \ref{point-line}, we have 
\[T_1\ll \frac{|D|^4|A|^2}{p}+p^{1/2}|A||D|^{21/8}.\]
So, 
\[\E\ll \frac{|D|^4|A|^4}{p}+p^{1/2}|A|^3|D|^{21/8}+|A|^2|D|^2.\]
Putting lower and upper bounds of $N$ together, we have 

\[|A|^8\ll |A^2+A^2|\cdot\left(\frac{|D|^4|A|^4}{p}+p^{1/2}|A|^3|D|^{21/8}\right).\]
If the first term dominates, we obtain 
\[|A^2+A^2||A-A|^4\gg p|A|^4,\]
otherwise, we have 
\[|A^2+A^2||A-A|^{21/8}\gg \frac{|A|^{5}}{p^{1/2}}.\]

Hence, if $|A-A|=K|A|$, then we have 
\[|A^2+A^2|\gg \min \left\lbrace \frac{p}{K^4}, \frac{|A|^{19/8}}{K^{21/8}p^{1/2}}\right\rbrace.\]
In other words, when $K\sim 1$, we have 
\begin{enumerate}
\item if $|A|\gg p^{12/19}$, then we have $|A^2+A^2|\gg p$.
\item if $p^{\frac{4}{7}+\epsilon}\le |A|\le p^{\frac{2}{3}-\epsilon}$ for some $\epsilon>0$, then $|A^2+A^2|\gg |A|^{\frac{3}{2}+\epsilon'}$ with $\epsilon'=\epsilon'(\epsilon)>0$. 
\end{enumerate}
{\bf Case $2$:} Assume $|D|\ll p^{5/8}$. If we use the estimate (\ref{eqener2}), namely, 
\[\sum_{\ell\in \overline{L}_B, ~\ell ~\mathtt{non-isotropic}}m(\ell)^2\ll p^{1/3}|D|^{14/3},\]
then the same argument gives us 
\[|A|^8\ll |A^2+A^2|\cdot \bigg(\frac{|D|^4|A|^4}{p}+p^{1/2}|A|^3p^{1/6}|D|^{7/3}\bigg).\]
This estimate tells us that 
\[|A^2+A^2|\gg \min \left\lbrace \frac{p}{K^4}, \frac{|A|^{8/3}}{K^{7/3}p^{2/3}}\right\rbrace.\]
Therefore, when $K\sim 1$, 
\begin{enumerate}
\item if $|A|\sim p^{5/8}$, then we have $|A^2+A^2|\gg p$.
\item if $p^{\frac{4}{7}+\delta}\le |A|\le p^{\frac{2}{3}-\delta}$ for some $\delta>0$, then $|A^2+A^2|\gg |A|^{\frac{3}{2}+\delta'}$ with $\delta'=\delta'(\delta)>0$. 
\end{enumerate}
\section{Variants of Theorem \ref{thm2}}
In this section, we discuss variants of Theorem \ref{thm2}, which will be obtained by using different bounds for the number of isosceles triangles $T$ (the same notation as above) in the proof of Theorem \ref{thm2}. Compared to lower bounds of Theorem \ref{thm2}, we observe that all of them are weaker.  For simplicity, we only consider the case $|A-A|\sim |A|$. 

We recall from the previous section that $T_1$ is the number of isosceles triangles $(x, y, z)\in (-A\times -A)\times (D\times D)\times (D\times D)$ such that $||x-y||=||x-z||\ne 0$ and $||y-z||\ne 0$.
\subsection{Bounding $T_1$ via a point-line incidence bound for small sets}
Let us first recall the following variant of a point-line incidence bound due to Stevens and De Zeeuw stated in \cite{nhopt}.
\begin{theorem}\label{lienthuoc}
Let $A$ be a set in $\mathbb{F}_p$ and $L$ a set of lines in $\mathbb{F}_p^2$. The number of incidences between $A\times A$ and $L$ is bounded by 
\[I(A\times A, L)\le \frac{|A|^{3/2}|L|}{p^{1/2}}+ |A|^{5/4}|L|^{3/4}+|A|^2+|L|.\]
\end{theorem}
Using an argument which is similar to that of \cite[Proof of Lemma 15]{LP}, we have the following result. 
\begin{lemma} Let $Q:=\sum_{\ell \in \overline{L}_B,~\ell ~\mathtt{non-isotropic}}m(\ell)^2$. We have
\[T_1\lesssim \frac{|A|^{3/2}|D|^4}{p^{1/2}}+ |A|^{5/4}|Q|^{1/4}|D|^{2}+|D|^4+|A|^2|D|^2.\]
\end{lemma}

\begin{proof}
For $k\ge 1$, let $L_k$ be the set of distinct non-isotropic lines in $\overline{L}_B$ with multiplicity between $k$ and $2k$.  We observe
\[|D|^4\ge\sum_{~\ell ~\mathtt{non-isotropic}}m(\ell)\ge k|L_k|,\]
and 
\[Q=\sum_{~\ell ~\mathtt{non-isotropic}}m(\ell)^2\ge k^2|L_k|.\]
Thus, 

\begin{align*}
 T_1&=\sum_{~\ell ~\mathtt{non-isotropic}}m(\ell)i(\ell)<\sum_{i}\sum_{\ell\colon ~2^i\le m(\ell)<2^{i+1}}2^{i+1}\cdot i(l)=\sum_{i}2^{i+1}\cdot I(-A\times -A, L_{2^i})\\
 &=\sum_{i, ~2^{i+1}\le \frac{Q}{|D|^4}}2^{i+1}\cdot I(-A\times -A, L_{2^i})+\sum_{i, ~2^{i+1}>\frac{Q}{|D|^4}}2^{i+1}\cdot I(-A\times -A,  L_{2^i})\\
 &=I+II. 
\end{align*}
Using $|L_k|\le |D|^4/k$ and Theorem \ref{lienthuoc}, one has
\begin{align*}
I&\lesssim \frac{|A|^{3/2}|D|^4}{p^{1/2}}+\sum_{i, ~2^{i+1}\le \frac{Q}{|D|^4}}2^{i+1}\cdot \left(|A|^{5/4}\left(\frac{|D|^4}{2^i}\right)^{3/4}+|A|^2+|L_{2^i}|\right)\\
&\lesssim \frac{|A|^{3/2}|D|^4}{p^{1/2}}+ |A|^{5/4}|D|^2Q^{1/4}+|A|^2|D|^2+|D|^4.\\
\end{align*}
Similarly, using $|L_k|\le Q/k^2$, we have 
\begin{align*}
II&\lesssim \frac{|A|^{3/2}|D|^4}{p^{1/2}}+\sum_{i, ~2^{i+1}> \frac{Q}{|D|^4}}2^{i+1}\cdot \left(|A|^{5/4}\left(\frac{Q}{2^{2i}}\right)^{3/4}+|A|^2+|L_{2^i}|\right)\\
&\lesssim \frac{|A|^{3/2}|D|^4}{p^{1/2}}+ |A|^{5/4}|D|^2Q^{1/4}+|A|^2|D|^2+|D|^4.\\
\end{align*}
In other words, 
\[T_1\lesssim \frac{|A|^{3/2}|D|^4}{p^{1/2}}+ |A|^{5/4}|Q|^{1/4}|D|^{2}+|D|^4+|A|^2|D|^2.\]
\end{proof}

We now follow the proof of Theorem \ref{thm2}. 

{\bf Case $1$:} If $|A-A|\sim |A|\ll p^{2/3}$, then $Q\le |D|^{21/4}$. Hence
\[|A|^8\lesssim |A^2+A^2||A|^2\left(\frac{|A|^{3/2}|D|^4}{p^{1/2}}+|D|^4+|A|^2|D|^2+|A|^{5/4}|D|^2|D|^{21/16}\right).\]
This implies
\[|A^2+A^2|\gtrsim \min \left\lbrace |A|^{\frac{23}{16}}, |A|^{1/2}p^{1/2} \right\rbrace.\]
{\bf Case $2$:} If $|A-A|\sim |A|\ll p^{5/8}$, then $Q\le p^{1/3}|D|^{14/3}$. We get
\[|A|^8\lesssim |A^2+A^2||A|^2\left(\frac{|A|^{3/2}|D|^4}{p^{1/2}}+|D|^4+|A|^2|D|^2+p^{1/12}|A|^{5/4}|D|^2|D|^{14/12}\right).\]
This implies
\[|A^2+A^2|\gtrsim \min \left\lbrace \frac{|A|^{\frac{19}{12}}}{p^{1/12}}, |A|^{1/2}p^{1/2} \right\rbrace.\]
\subsection{Bounding $T_1$ via Rudnev's point-plane incidence bound}
Instead of using the bisector energy for lines in $L_B$ and the point-line incidence bound in Theorem \ref{point-line}, we can bound $T_1$ directly by using Rudnev's point-plane incidence bound \cite{RR} as Petridis did in  \cite{petridis}. 

More precisely,  we can follow his proof identically to bound for the case of $T_1$, namely, we have 
\[T_1\ll \frac{|A|^2|D|^4}{p}+|A|^{3/2}|D|^3.\]
So, with the argument as in the proof of Theorem \ref{thm2}, one has
\begin{align*}
|A|^8&\le |A^2+A^2|\cdot |A|^2\cdot \left(\frac{|A|^2|D|^4}{p}+|A|^{3/2}|D|^3\right)\\
|A|^6&\le |A^2+A^2|\cdot\left(\frac{|A|^2|A-A|^4}{p}+|A|^{3/2}|A-A|^3  \right).
\end{align*}
If $|D|=|A-A|\sim |A|$, then
\[|A^2+A^2|\gg \min \left\lbrace p, |A|^{3/2}\right\rbrace.\]
\subsection{Bounding $T_1$ via the Cauchy-Schwarz inequality}
We use the fact that 
\[T_1=\sum_{\ell ~\mathtt{non-istropic}}i(\ell)m(\ell)\le \left(\sum_{\ell\in \overline{L}_B}i(\ell)^2\right)^{1/2}\cdot \left(\sum_{\ell ~\mathtt{non-istropic}}m(\ell)^2 \right)^{1/2},\]
where $i(\ell)$ is the number of points from $-A\times -A$ on the line $\ell$. 
Thus, 
\[T_1\le |A|^2\cdot \left(\sum_{\ell ~\mathtt{non-istropic}}m(\ell)^2 \right)^{1/2}.\]
{\bf Case $1$:} If $|A-A|\sim |A|\ll p^{2/3}$, then using (\ref{eqener1}), we have 
$T_1\ll |A|^{\frac{37}{8}}$. This gives us 
\[|A|^8\ll |A^2+A^2|\cdot |A|^2\cdot |A|^{37/8},\]
so $|A^2+A^2|\gg |A|^{11/8}$. 

{\bf Case $2$:} If $|A-A|\sim |A|\ll p^{5/8}$, then using (\ref{eqener2}), we have 
\[T_1\le |A|^2\cdot |A|^{7/3}\cdot p^{1/6}.\]
Hence, 
\[|A|^8\ll |A^2+A^2|\cdot |A|^2\cdot |A|^{2+\frac{7}{3}}\cdot p^{1/6},\]
so $|A^2+A^2|\gg \frac{|A|^{5/3}}{p^{1/6}}$. 
\section{Proofs of Theorems \ref{thm1.333} and \ref{thm1.444}}

\begin{proof}[Proof of Theorem \ref{thm1.333}]
This proof is very similar to that of Theorem \ref{thm1}. Let $\lambda$ be an arbitrary element in $\mathbb{F}_p$. To obtain the desired result, it is enough to show that the following equation has at least one solution:
\begin{equation} \label{eqcoro} (x-y)^2+u+v=\lambda, \end{equation}
where $x,y \in A,$ $u,v \in A^2+A^2$.

Let $P=\{(-2x,v+x^2):x \in A, v \in A^2+A^2 \}$ be a point set in $\mathbb{F}_p^2$. Let $L$ be the set of lines of the form 
$$yX +Y = \lambda -u-y^2, $$
where $y \in A$, $u \in A^2+A^2$.
The number of solutions of (\ref{eqcoro}) is equal to the number of incidences between $P$ and $L$ in $\mathbb{F}_p^2$. By Theorem \ref{ptsplane}, if $|P||L| \gg p^3$, then there is at least one incidence between $P$ and $L$. We also have that $|P|=|L|=|A||A^2+A^2|$. We now need to verify the condition 
$$|A|^2\cdot |A^2+A^2|^2 \gg p^3. $$
Because $|A-A| \sim |A|$, we can use the second bound of Theorem \ref{thm2}, and verifying both cases for the lower bound of $|A^2+A^2|$, we obtain the condition $|A| \gg p^{13/22}$.
\end{proof}
To prove Theorem \ref{thm1.444}, we need a point-plane incidence bound in $\mathbb{F}_p^3$. 
\begin{theorem}[Theorem 5, \cite{vinhline}] \label{thmV}
Let $P$ be a set of points and $H$ be a set of planes in $\mathbb{F}_p^3$. Then the number of incidences between $P$ and $H$, denoted by $I(P, H)$, satisfies
\[\left\vert I(P, H)-\frac{|P||H|}{p} \right\vert \le p\sqrt{|P||H|}.\]
\end{theorem}
\begin{proof}[Proof of Theorem \ref{thm1.444}]
Let $\lambda$ be an arbitrary element of $\mathbb{F}_p$. We need to show that the following equation has at least one solution
\begin{equation} \label{eqcoro2} (x-y)^2+(s-t)^2+u+v=\lambda, \end{equation}
where $x,y,s,t \in A$ and $u,v \in A^2+A^2$.

Let $P=\{(-2x,-2s,v+x^2+s^2):x \in A,~ s \in A,~ v \in A^2+A^2\}$ be a point set in $\mathbb{F}_p^3$ and $H$ the set of planes of the form
$$yX+tY+Z= \lambda-u-y^2-t^2, $$
where $y,t \in A,$  $u \in A^2+A^2$.
One can see that the number of solutions of (\ref{eqcoro2}) is equal to the number of incidences between $P$ and $H$ in $\mathbb{F}_p^3$. Using Theorem \ref{thmV}, we obtain at least one incidence between $P$ and $H$ when $|P||H| \gg p^4$. Given our choice of $P$ and $H$, we have $|P|=|A|^2|A^2+A^2|=|H|$.
With our hypothesis $|A-A| \sim |A|$, we can use an appropriate bound on $A^2+A^2$ in Theorem $\ref{thm2}$ (the first bound is sufficient for the result, but the second bound for $p^{4/7} \ll |A| \leq p^{5/8}$ leads to the same result), to obtain the condition $|A| \gg p^{4/7}$. Therefore equation (\ref{eqcoro2}) has at least one solution when $|A| \gg p^{4/7}$, which is the desired result.
\end{proof}

\section*{Acknowledgments}
The second listed author was supported by Swiss National Science Foundation grant P4P4P2-191067.

\end{document}